\documentclass
 {amsart}
\usepackage{newunicodechar}
\newunicodechar{，}{,}

\usepackage{amsmath,cite,amsfonts,amssymb,amsthm,amscd,latexsym,tikz-cd}
\usetikzlibrary{arrows.meta}

\usepackage{array}
\usepackage{mathrsfs}
\usepackage{color}
\usepackage{comment}
 \usepackage{lmodern}
\usepackage[T1]{fontenc}
\usepackage{a4wide}
\usepackage{amsfonts}
\usepackage{amssymb,dsfont}
\usepackage{xcolor}
\usepackage{tasks}
\usepackage{longtable}

\newtheorem{thm}{Theorem}[section]
\newtheorem{lem}[thm]{Lemma}
\newtheorem{cor}[thm]{Corollary}
\newtheorem{pro}[thm]{Proposition}

\theoremstyle{definition}
\newtheorem{defn}{Definition}[section]

\newtheorem{example}{Example}[section]
\theoremstyle{remark}
\newtheorem{rmk}{Remark}[section]

\begin{document}

\title{Poisson $n$-Lie algebras:  constructions \\ and the structure of solvable algebras}

\author[Xinru Cao]{Xinru Cao}
\address{Xinru Cao \newline \indent School of Mathematics and Statistics, 
Northeast Normal University, Changchun,  130024, China.}
\email{{\tt caoxinru@nenu.edu.cn}}

\author[Zafar Normatov]{Zafar Normatov}
\address{Zafar Normatov \newline \indent School of Mathematics, Jilin University, Changchun, 130012, China,\newline \indent
Institute of Mathematics, Uzbekistan Academy of Sciences, Tashkent, 100174, Uzbekistan} \email{z.normatov@inbox.ru}

\author[Bakhrom Omirov]{Bakhrom Omirov}
\address{Bakhrom A. Omirov. \newline \indent 
Institute for Advanced Study in Mathematics, Harbin Institute of Technology, Harbin 150001,\newline \indent 
Suzhou Research Institute, Harbin Institute of Technology, Suzhou 215104, China.}
\email{{\tt omirovb@mail.ru}}

\begin{abstract} 
In this paper, we develop a construction of Poisson $n$-Lie algebras that generalizes the Jacobian $n$-Lie construction. Using the Grassmann--Pl\"ucker relations, we derive necessary and sufficient conditions under which the resulting bracket defines a Poisson $n$-Lie algebra. We also prove that suitable quotients of tensor products of Poisson algebras carry natural Poisson $n$-Lie structures. Conversely, we give a tensor-type procedure that associates a Poisson algebra to a given Poisson $n$-Lie algebra. The quotient and converse constructions thus provide two systematic methods for relating Poisson algebras to Poisson $n$-Lie algebras. 

We further establish analogues of Engel's theorem and Lie's theorem and characterize solvability and nilpotency of Poisson $n$-Lie algebras in terms of their underlying associative and $n$-Lie structures. We introduce hypo-nilpotent ideals and investigate maximal such ideals in finite-dimensional solvable Poisson $n$-Lie algebras. Finally, we prove that the generalized eigenspaces of multiplication operators are ideals.

\end{abstract}



\maketitle

\numberwithin{equation}{section}


\section{Introduction}

Motivated by Kurosh's general notion of an $\Omega$-algebra~\cite{Kurosh}, Filippov introduced $n$-Lie algebras as higher-arity generalizations of Lie algebras~\cite{Filippov1}; these structures are also commonly referred to as Filippov algebras. In the infinite-dimensional setting, he constructed fundamental examples whose brackets are defined by Jacobian determinants~\cite{Filippov2}.

A different motivation for the theory originates in Nambu's pioneering work~\cite{Nambu} on a generalization of classical Hamiltonian mechanics. Motivated by Liouville's theorem, Nambu extended Hamiltonian dynamics to a three-dimensional phase space in which the equations of motion are governed by two Hamiltonian functions and a ternary bracket involving three canonical variables.

Takhtajan subsequently developed the foundations of a canonical formalism for Nambu mechanics~\cite{Takh}, thereby placing Nambu's generalization of Hamiltonian mechanics on a firm geometric and algebraic footing. In this framework, the Nambu bracket leads naturally to the notion of a Nambu--Poisson manifold: the classical binary Poisson bracket on observables is replaced by an $n$-ary multilinear bracket, where $n\geq 3$. The resulting notion of a Nambu algebra of order $n$ provides a natural higher-arity generalization of a Lie algebra and is equivalent to that of an $n$-Lie algebra. More generally, the notion of a Nambu--Leibniz algebra of order $n$ was introduced by relaxing skew-symmetry. The same structure was subsequently studied in~\cite{Casas} under the name Leibniz $n$-algebra, reflecting the fact that, when $n=2$, it reduces to a Leibniz algebra in the sense of Loday~\cite{Loday}.

Since their introduction, $n$-Lie algebras and related higher-arity algebraic structures have attracted considerable attention. Their theory has been developed extensively in, among others,~\cite{Goze,Kasymov,Ling,Makhlouf,Pozhidaev1,Pozhidaev2} and the references therein.

A fundamental example of an $n$-Lie algebra is the Jacobian $n$-Lie algebra, which naturally carries a compatible Poisson $n$-Lie algebra structure. In fact, Cantarini and Kac~\cite{CK2} proved that, for $n>2$, every simple linearly compact Poisson $n$-Lie algebra is isomorphic to $\mathbb{F}[[x_1,\ldots,x_n]]$ equipped with the $n$-ary bracket defined by the Jacobian determinant. When $n=2$, the same construction yields an ordinary Poisson algebra.

A comprehensive overview of various $n$-ary generalizations of Lie algebras, including $n$-Lie algebras, $n$-Leibniz algebras, and generalized Poisson $n$-Lie structures, together with their applications to mathematical physics and cohomology, is presented in~\cite{De}. We also refer to~\cite{Mashurov} for notable results concerning simple Poisson $n$-Lie algebras.

In the present paper, we generalize the Jacobian construction by adjoining an arbitrary number $m\geq 1$ of additional columns whose entries are fixed elements of a commutative associative algebra. Using the Grassmann--Pl\"ucker relations for the complementary minors, we derive necessary and sufficient algebraic conditions under which the resulting determinant bracket defines a Poisson $n$-Lie algebra. This construction provides a systematic method for producing new examples of Poisson $n$-Lie algebras.

We further prove that, if $P$ is a Poisson $n$-Lie algebra and $B$ is a commutative associative algebra, then the tensor product $P\otimes B$ carries a natural Poisson $n$-Lie algebra structure.

We next consider the tensor product of two Poisson algebras equipped with an $n$-ary bracket obtained by iterating their underlying binary Poisson brackets in the manner introduced in~\cite{Xu}. By factoring out the ideal generated by the elements measuring the failure of skew-symmetry, we obtain a quotient carrying a Poisson $n$-Lie algebra structure. In particular, given a Poisson algebra, the corresponding quotient of its tensor product with an isomorphic copy is a Poisson $n$-Lie algebra. Conversely, we give a tensor-type procedure for constructing a Poisson algebra from a given Poisson $n$-Lie algebra. Taken together, these constructions provide two systematic methods for relating Poisson algebras to Poisson $n$-Lie algebras.

We also investigate the structure of solvable Poisson $n$-Lie algebras. In particular, we establish analogues of Engel's theorem and Lie's theorem in this setting and derive several consequences. We further characterize solvability in terms of the underlying $n$-Lie and commutative associative structures. In addition, we introduce hypo-nilpotent ideals and study maximal hypo-nilpotent ideals in finite-dimensional solvable Poisson $n$-Lie algebras.

Finally, we prove that the generalized eigenspaces of multiplication operators in commutative associative algebras are ideals. We also obtain results concerning idempotent elements that extend naturally from ordinary Poisson algebras to the $n$-ary setting.
\section{Preliminaries}

In this section, we recall the definitions and examples that will be used throughout the paper.

A Poisson $n$-Lie algebra carries both  an associative commutative product and an $n$-Lie bracket. We therefore begin by recalling Filippov's definition of an $n$-Lie algebra~\cite{Filippov1}.

\begin{defn}
Let $\mathbb{F}$ be a field. An $\mathbb{F}$-vector space $\mathcal{L}$ equipped with an $n$-linear skew-symmetric bracket $[-,\ldots,-]$ is called an \emph{$n$-Lie algebra} if it satisfies the fundamental identity
\begin{equation}\label{eq2.1}
[x_1,\ldots,x_{n-1},[y_1,\ldots,y_n]]=\sum_{i=1}^{n}[y_1,\ldots,y_{i-1},[x_1,\ldots,x_{n-1},y_i],y_{i+1},\ldots,y_n]
\end{equation}
for all $x_1,\ldots,x_{n-1},y_1,\ldots,y_n\in\mathcal{L}$.
\end{defn}

For examples of $n$-Lie algebras, we refer the reader to~\cite{CK1,Dj,Filippov2,newexample}. Among the standard examples, the following construction is particularly relevant to the present paper.

\begin{example}\label{exam2.1} 
Let $(\mathcal{A}, \cdot)$ be a commutative associative algebra, and let $d_{1}, \dots, d_{n}$ be pairwise commuting derivations of $\mathcal{A}$. We define an $n$-ary bracket on $\mathcal{A}$ by
$$[x_{1}, \ldots, x_{n}]:= \operatorname{Jac}(x_{1}, \ldots, x_{n})=\left|
\begin{array}{ccc}
d_{1}(x_{1}) & \cdots & d_{1}(x_{n}) \\
\vdots & \ddots & \vdots \\
d_{n}(x_{1}) & \cdots & d_{n}(x_{n})
\end{array}
\right| \quad \mbox{for all} \quad x_{1},\ldots,x_{n} \in \mathcal{A}.$$
Then $(\mathcal{A},[-,\ldots,-])$ is an $n$-Lie algebra, called the \emph{Jacobian $n$-Lie algebra} associated with $d_1,\ldots,d_n$. Moreover, its bracket satisfies the Leibniz rule
\begin{equation}\label{eq2.2}
[y\cdot z,x_2,\ldots,x_n]=y\cdot[z,x_2,\ldots,x_n]+z\cdot[y,x_2,\ldots,x_n]
\end{equation}
for all $y,z,x_2,\ldots,x_n\in\mathcal{A}$.
\end{example}

We now recall the definition of a Poisson $n$-Lie algebra.

\begin{defn}
An algebra $(\mathcal{P},\cdot,[-,\ldots,-])$ is called a \emph{Poisson $n$-Lie algebra} if $(\mathcal{P},\cdot)$ is a commutative associative algebra, $(\mathcal{P},[-,\ldots,-])$ is an $n$-Lie algebra, and the bracket satisfies the Leibniz rule \eqref{eq2.2} with respect to the associative product.
\end{defn}

Consequently, Example~\ref{exam2.1} defines a Poisson $n$-Lie algebra. The Jacobian construction was generalized in~\cite{newexample} to produce further examples. More precisely, one starts with  a commutative associative algebra $\mathcal{A}$ equipped with $n+1$ pairwise commuting derivations and defines a skew-symmetric $n$-ary bracket by adjoining to the corresponding Jacobian matrix a fixed additional column whose entries belong to $\mathcal{A}$. Necessary and sufficient conditions on the entries of this column ensuring that the resulting bracket satisfies the fundamental identity \eqref{eq2.1} were established. Under these conditions, the construction yields new $n$-Lie algebras that are naturally endowed with compatible Poisson $n$-Lie algebra structures.

We conclude the preliminaries with a Grassmann--Pl\"ucker relation that will be used in the determinant construction below.

Assume that \(E\) is a spanning subset of the vector space
\(\mathbb{F}^{n-1}\), and regard the elements of \(E\) as row vectors. Then the \emph{Grassmann--Plücker relation} takes the form 
\begin{equation}\label{eq2.3}
\sum_{k=1}^{n}(-1)^{k-1}
\det\bigl(c_{1},\ldots,\widehat{c_{k}},\ldots,c_{n}\bigr)
\det\bigl(c_{k},d_{3},\ldots,d_{n}\bigr)=0
\end{equation}
for all $c_{1},\ldots,c_{n},d_{3},\ldots,d_{n}\in E.$ 
See \cite{Plucker1,Plucker2} for further background on the
Grassmann--Plücker relations.

\section{Determinant-type Construction of Poisson $n$-Lie Algebras}\label{sec:determinantal-construction}

Let $n\geq 2$ and $m\geq 1$, and let $(\mathcal A,\cdot)$ be a commutative associative algebra over a field $\mathbb F$ of characteristic zero. Suppose that $\mathcal A$ is equipped with pairwise commuting $\mathbb F$-linear derivations 
$$d_1,\ldots,d_{n+m}\colon\mathcal A\longrightarrow\mathcal A.$$
Fix a matrix $A=(a_{r,s})\in M_{n+m,m}(\mathcal A)$. For
$x_1,\ldots,x_n\in\mathcal A$, define
\begin{equation}\label{eq:determinant-bracket}
[x_1,\ldots,x_n]=
\det
\begin{pmatrix}
d_1(x_1) & \cdots & d_1(x_n) & a_{1,1} & \cdots & a_{1,m}\\
d_2(x_1) & \cdots & d_2(x_n) & a_{2,1} & \cdots & a_{2,m}\\
\vdots & \vdots & \vdots & \vdots & \ddots & \vdots\\
d_{n+m}(x_1) & \cdots & d_{n+m}(x_n)
& a_{n+m,1} & \cdots & a_{n+m,m}
\end{pmatrix}.
\end{equation}

Write $N_{n+m}=\{1,\ldots,n+m\}$. For an $n$-element subset $I=\{i_1<\cdots<i_n\}\subseteq N_{n+m}$, let $I^c=N_{n+m}\setminus I$ and define 
$$\pi^I:=(-1)^{i_1+\cdots+i_n+\frac{n(n+1)}2}\det(A_{I^c}),$$
where $A_{I^c}$ is the $m\times m$ submatrix of $A$ whose rows are indexed by $I^c$ and written in increasing order. We extend this notation to ordered indices by alternation. Thus, for every $\sigma\in S_n$, set
$$\pi^{(i_{\sigma(1)},\ldots,i_{\sigma(n)})}:=\operatorname{sgn}(\sigma)\pi^I,$$
and set $\pi^{(i_1,\ldots,i_n)}=0$ whenever two indices coincide. For an ordered $n$-tuple
$\mathbf i=(i_1,\ldots,i_n)\in N_{n+m}^n$, we abbreviate $\pi^{(i_1,\ldots,i_n)}$ to $\pi^{\mathbf i}$.

The definition of $\pi$ includes all signs from expanding the determinant and changing the order of the indices. If $I=\{i_1 < \cdots < i_n\}$, then
\begin{align*}
[x_1,\ldots,x_n]
&=\sum_{|I|=n}\pi^I\det\bigl(d_{i_p}(x_q)\bigr)_{1\le p,q\le n}\\
&=\sum_{|I|=n}\sum_{\sigma\in S_n}\operatorname{sgn}(\sigma)\pi^I\prod_{p=1}^{n}d_{i_{\sigma(p)}}(x_p).
\end{align*}
For a fixed set $I$, the ordered $n$-tuples $(i_{\sigma(1)},\ldots, i_{\sigma(n)})$, with $\sigma\in S_n$, are exactly the orderings of the elements of $I$. 
Since each ordered tuple with distinct entries occurs exactly once and tuples with repeated entries may be added without changing the sum because their coordinates are zero, we therefore obtain the single formula
\begin{equation}\label{expand}
[x_1,\ldots,x_n]=\sum_{\mathbf i\in N_{n+m}^n}\pi^{\mathbf i}\prod_{p=1}^{n}d_{i_p}(x_p).
\end{equation}

The following identity for the signed complementary minors is a consequence of the Grassmann--Pl\"ucker relations. 

\begin{pro} 
For all $\mathbf i=(i_1,\ldots,i_n)$ and $\mathbf j=(j_1,\ldots,j_n)$ in $N_{n+m}^n$, one has
\begin{equation}\label{eq:plucker-exchange}
\sum_{k=1}^{n}\pi^{(j_1,\ldots,j_{k-1},i_1,j_{k+1},\ldots,j_n)}\pi^{(j_k,i_2,\ldots,i_n)}=\pi^{\mathbf i}\pi^{\mathbf j}.
\end{equation}
\end{pro}

\begin{proof}
Both sides of \eqref{eq:plucker-exchange} are polynomial expressions with integer coefficients in the entries of $A$. It is therefore enough to prove the identity for the generic $(n+m)\times m$ matrix $X=(x_{rs})$ over the field
\[K=\mathbb{Q}\bigl(x_{rs}\mid 1\leq r\leq n+m,\ 1\leq s\leq m\bigr),\]
where the elements \(x_{rs}\) are algebraically independent over
\(\mathbb{Q}\).

The submatrix of $X$ formed by the first \(m\) rows is a generic
\(m\times m\) matrix. Its determinant is a nonzero polynomial in the
variables \(x_{rs}\), and hence it is a nonzero, therefore invertible,
element of \(K\). Thus \(X\) has an invertible \(m\times m\) minor.
Since \(X\) has exactly \(m\) columns, it follows that
\[\operatorname{rank}_{K}X=\operatorname{rank}_{K}X^{\mathsf{T}}=m.\]
Write the columns of $X^{\mathsf T}$ as $v_1,\ldots,v_{n+m}\in K^m$ and define 
$$\varphi(q_1,\ldots,q_m):=\det(v_{q_1},\ldots,v_{q_m}),\qquad q_1,\ldots,q_m\in N_{n+m}.$$
Since $\operatorname{rank}X^{\mathsf T}=m$, at least one $m\times m$ minor of $X^{\mathsf T}$ is nonzero, so $\varphi$ is not identically zero. Moreover, the determinant is alternating, and hence
$$\varphi(\ldots,q_i,\ldots,q_j,\ldots)=-\varphi(\ldots,q_j,\ldots,q_i,\ldots),$$
and $\varphi(q_1,\ldots,q_m)=0$ whenever two indices coincide and the maximal minors of a full-rank matrix satisfy \eqref{eq2.3}. Thus, $\varphi$ is a rank-$m$ Grassmann--Pl\"ucker function in the sense of \cite[Definition~3.15]{BakerBowler2019}.

Let
$$I=\{i_1<\cdots<i_n\},\qquad I^c=\{r_1<\cdots<r_m\}.$$
The columns $r_1,\ldots,r_m$ of $X^{\mathsf T}$ form the matrix $(X_{I^c})^{\mathsf T}$, and therefore
$$\varphi(r_1,\ldots,r_m)=\det\bigl((X_{I^c})^{\mathsf T}\bigr)=\det(X_{I^c}).$$

Consider the dual weak Grassmann--Plücker function on \(I\), defined by
\[\varphi^{*}(i_{1},\ldots,i_{n})=\operatorname{sgn}(i_{1},\ldots,i_{n},r_{1},\ldots,r_{m})
\varphi(r_{1},\ldots,r_{m}),\]
where \(r_{1},\ldots,r_{m}\) are the elements of
\(I\setminus\{i_{1},\ldots,i_{n}\}\) (see for the details in \cite[Definition~4.1]{BakerBowler2019}).

Since both $i_1<\cdots<i_n$ and $r_1<\cdots<r_m$ are increasing, the inversions in the permutation
$$(i_1,\ldots,i_n,r_1,\ldots,r_m)$$
arise precisely from pairs $(i_q,r_p)$ with $r_p<i_q$. For each $i_q$, there are exactly $i_q-q$ such elements of $I^c$. Hence
\begin{align*}
\operatorname{sgn}(i_1,\ldots,i_n,r_1,\ldots,r_m)
&=(-1)^{\sum_{q=1}^{n}(i_q-q)}\\
&=(-1)^{i_1+\cdots+i_n-\frac{n(n+1)}2}.
\end{align*}
Consequently, we get 
\begin{align*}
\varphi^{*}(i_{1},\ldots,i_{n})&=(-1)^{i_1+\cdots+i_n+\frac{n(n+1)}2}\det(X_{I^c})\\
&=\pi^I.
\end{align*}
By \cite[Lemma~4.2]{BakerBowler2019} $\varphi^{*}$ satisfies 
the Grassmann--Pl\"ucker relation given in \cite[Definition~3.15]{BakerBowler2019}. Therefore, for $\pi$ the relation has the following form 
\begin{equation}\label{eq3.4.1}
\sum_{k=1}^{n+1}(-1)^k\pi^{(c_1,\ldots,\widehat{c_k},\ldots,c_{n+1})}\pi^{(c_k,d_1,\ldots,d_{n-1})}=0.
\end{equation}

Assuming 
$$j_p=c_p, \quad 1\leq p \leq n, \quad i_1=c_{n+1}, \quad 
i_{q+1}=d_q, \quad 1\leq q \leq n-1,$$
then \eqref{eq3.4.1} takes the form
$$0=\sum_{k=1}^{n}(-1)^k\pi^{(j_1,\ldots,\widehat{j_k},\ldots,j_n,i_1)}\pi^{(j_k,i_2,\ldots,i_n)}+(-1)^{n+1}\pi^{\mathbf j}\pi^{\mathbf i}.$$
For each $k\in\{1,\ldots,n\}$, moving $i_1$ from the last position to the $k$-th position requires $n-k$ transpositions. Therefore, by alternating property, we get 
$$\pi^{(j_1,\ldots,\widehat{j_k},\ldots,j_n,i_1)}=(-1)^{n-k}\pi^{(j_1,\ldots,j_{k-1},i_1,j_{k+1},\ldots,j_n)}.$$
Substituting this equality gives
\begin{align*}
0=&\sum_{k=1}^{n}(-1)^k(-1)^{n-k}\pi^{(j_1,\ldots,j_{k-1},i_1,j_{k+1},\ldots,j_n)}\pi^{(j_k,i_2,\ldots,i_n)}+(-1)^{n+1}\pi^{\mathbf j}\pi^{\mathbf i}\\
=&(-1)^n\Bigg(\sum_{k=1}^{n}\pi^{(j_1,\ldots,j_{k-1},i_1,j_{k+1},\ldots,j_n)}\pi^{(j_k,i_2,\ldots,i_n)}-\pi^{\mathbf i}\pi^{\mathbf j}\Bigg).
\end{align*}
Hence
$$\sum_{k=1}^{n}\pi^{(j_1,\ldots,j_{k-1},i_1,j_{k+1},\ldots,j_n)}\pi^{(j_k,i_2,\ldots,i_n)}=\pi^{\mathbf i}\pi^{\mathbf j}.$$
The same equality holds for arbitrary ordered tuples $\mathbf i,\mathbf j\in N_{n+m}^n$.

Thus \eqref{eq:plucker-exchange} holds for the generic matrix $X$. Let the difference between the two sides of \eqref{eq:plucker-exchange} be $F(x_{rs})\in\mathbb Z[x_{rs}]$. The preceding argument gives
$$F(x_{rs})=0\quad\text{in }K.$$
Since the natural inclusion $\mathbb Z[x_{rs}]\hookrightarrow K$ is injective, $F$ is the zero polynomial. Substituting $x_{rs}=a_{rs}\in\mathcal A$ therefore proves \eqref{eq:plucker-exchange} for the original matrix $A$.
\end{proof}

\begin{cor}\label{s_2(y_t)}
For all $\mathbf i,\mathbf j\in N_{n+m}^n$ and every $t\in\{2,\ldots,n\}$, one has
\begin{equation}\label{eq3.4}
\begin{aligned}
\sum_{k=1}^{n}\Bigg(&\pi^{(j_1,\ldots,j_{k-1},i_1,j_{k+1},\ldots,j_n)}\pi^{(j_k,i_2,\ldots,i_n)}\\
+&\pi^{(j_1,\ldots,j_{k-1},i_t,j_{k+1},\ldots,j_n)}\pi^{(j_k,i_2,\ldots,i_{t-1},i_1,i_{t+1},\ldots,i_n)}\Bigg) =0
\end{aligned}
\end{equation}
\end{cor}

\begin{proof}
Fix $t\in\{2,\ldots,n\}$ and set $\tau_t(\mathbf i)=(i_t,i_2,\ldots,i_{t-1},i_1,i_{t+1},\ldots,i_n).$ By \eqref{eq:plucker-exchange},
$$\sum_{k=1}^{n}\pi^{(j_1,\ldots,j_{k-1},i_1,j_{k+1},\ldots,j_n)}\pi^{(j_k,i_2,\ldots,i_n)}=\pi^{\mathbf i}\pi^{\mathbf j}.$$
Applying the same identity to $\tau_t(\mathbf i)$ gives
$$\sum_{k=1}^{n}\pi^{(j_1,\ldots,j_{k-1},i_t,j_{k+1},\ldots,j_n)}\pi^{(j_k,i_2,\ldots,i_{t-1},i_1,i_{t+1},\ldots,i_n)}=\pi^{\tau_t(\mathbf i)}\pi^{\mathbf j}.$$
Since $\tau_t$ interchanges the first and the $t$-th entries of $\mathbf i$ and $\pi$ is alternating, $\pi^{\tau_t(\mathbf i)}=-\pi^{\mathbf i}.$
Hence the two sums cancel, and \eqref{eq3.4} follows.
\end{proof}

The following results establish necessary and sufficient polynomial conditions for \((\mathcal A,\cdot,[-,\ldots,-]_A)\) to be a Poisson \(n\)-Lie algebra.

\begin{thm}[Determinantal criterion]\label{thm3.2}
The bracket in \eqref{eq:determinant-bracket} defines a Poisson $n$-Lie algebra structure on $\mathcal A$ if and only if
\begin{equation*}
\sum_{\mathbf i,\mathbf j\in N_{n+m}^n}\Bigg(\pi^{\mathbf i}d_{i_1}\bigl(\pi^{\mathbf j}\bigr)-\sum_{k=1}^{n}\pi^{(j_1,\ldots,j_{k-1},i_1,j_{k+1},\ldots,j_n)}d_{i_1}\bigl(\pi^{(j_k,i_2,\ldots,i_n)}\bigr)\Bigg)\prod_{q=2}^{n}d_{i_q}(y_q)\prod_{p=1}^{n}d_{j_p}(x_p)=0
\end{equation*}
for all $x_1,\ldots,x_n,y_2,\ldots,y_n\in\mathcal A$.
\end{thm}

\begin{proof}
The determinant in \eqref{eq:determinant-bracket} is alternating in its first $n$ columns, and the derivation property of the $d_r$ implies the Leibniz rule in each argument. It remains to determine exactly when the fundamental identity holds. By skew-symmetry, \eqref{eq2.1} is equivalent to
\begin{equation}\label{eq3.6}
[[x_1,\ldots,x_n],y_2,\ldots,y_n]=\sum_{k=1}^{n}[x_1,\ldots,x_{k-1},[x_k,y_2,\ldots,y_n],x_{k+1},\ldots,x_n].
\end{equation}
Denote the two sides of \eqref{eq3.6} by $\mathrm{LHS}$ and $\mathrm{RHS}$, respectively.

Applying \eqref{expand} first to the outer bracket and then to the inner bracket gives
$$\mathrm{LHS}=\sum_{\mathbf i,\mathbf j\in N_{n+m}^n}\pi^{\mathbf i}d_{i_1}\Bigg(\pi^{\mathbf j}\prod_{p=1}^{n}d_{j_p}(x_p)\Bigg)\prod_{q=2}^{n}d_{i_q}(y_q).$$
The Leibniz rule yields
$$d_{i_1}\Bigg(\pi^{\mathbf j}\prod_{p=1}^{n}d_{j_p}(x_p)\Bigg)=d_{i_1}\bigl(\pi^{\mathbf j}\bigr)\prod_{p=1}^{n}d_{j_p}(x_p)+\pi^{\mathbf j}\sum_{t=1}^{n}d_{i_1}\bigl(d_{j_t}(x_t)\bigr)\prod_{\substack{1\leq p\leq n\\p\ne t}}d_{j_p}(x_p).$$
Consequently,
\begin{equation}\label{eq3.7}
\begin{aligned}
\mathrm{LHS}={}&\sum_{\mathbf i,\mathbf j\in N_{n+m}^n}\pi^{\mathbf i}d_{i_1}\bigl(\pi^{\mathbf j}\bigr)\prod_{p=1}^{n}d_{j_p}(x_p)\prod_{q=2}^{n}d_{i_q}(y_q)\\
&+\sum_{t=1}^{n}\sum_{\mathbf i,\mathbf j\in N_{n+m}^n}\pi^{\mathbf i}\pi^{\mathbf j}d_{i_1}\bigl(d_{j_t}(x_t)\bigr)\prod_{\substack{1\leq p\leq n\\p\ne t}}d_{j_p}(x_p)\prod_{q=2}^{n}d_{i_q}(y_q).
\end{aligned}
\end{equation}
In \eqref{eq3.7}, the first sum contains the derivatives of the coefficients $\pi^{\mathbf j}$, while the second contains all second derivatives of the $x_t$. No second derivatives of the $y_q$ occur on the left-hand side.

We now expand the right-hand side. For a fixed $k\in\{1,\ldots,n\}$, formula \eqref{expand} gives
$$
\begin{aligned}
&[x_1,\ldots,x_{k-1},[x_k,y_2,\ldots,y_n],
x_{k+1},\ldots,x_n]=
\sum_{\mathbf j\in N_{n+m}^n}
\pi^{\mathbf j}
\prod_{\substack{1\leq p\leq n\\p\ne k}}d_{j_p}(x_p)
d_{j_k}\bigl([x_k,y_2,\ldots,y_n]\bigr),
\end{aligned}
$$
where
$$[x_k,y_2,\ldots,y_n]=\sum_{\mathbf i\in N_{n+m}^n}\pi^{\mathbf i}d_{i_1}(x_k)\prod_{q=2}^{n}d_{i_q}(y_q).$$
When $d_{j_k}$ is applied to this last expression, it acts on the coefficient $\pi^{\mathbf i}$, on $d_{i_1}(x_k)$, or on one of the factors $d_{i_t}(y_t)$. Hence
\begin{align*}
d_{j_k}\bigl([x_k,y_2,\ldots,y_n]\bigr)=& \quad \sum_{\mathbf i\in N_{n+m}^n}d_{j_k}\bigl(\pi^{\mathbf i}\bigr)d_{i_1}(x_k)\prod_{q=2}^{n}d_{i_q}(y_q)\\
&+\sum_{\mathbf i\in N_{n+m}^n}\pi^{\mathbf i}d_{j_k}\bigl(d_{i_1}(x_k)\bigr)\prod_{q=2}^{n}d_{i_q}(y_q)\\
&+\sum_{t=2}^{n}\sum_{\mathbf i\in N_{n+m}^n}\pi^{\mathbf i}d_{i_1}(x_k)d_{j_k}\bigl(d_{i_t}(y_t)\bigr)\prod_{\substack{2\leq q\leq n\\q\ne t}}d_{i_q}(y_q).
\end{align*}
Substituting this formula and summing over $k$ gives the complete
expansion
\begin{equation}\label{eq:FI-right-full-expansion}
\begin{aligned}
\mathrm{RHS}=
&\quad\sum_{k=1}^{n}\sum_{\mathbf i,\mathbf j\in N_{n+m}^n}\pi^{\mathbf j}d_{j_k}\bigl(\pi^{\mathbf i}\bigr)d_{i_1}(x_k)\prod_{\substack{1\leq p\leq n\\p\ne k}}d_{j_p}(x_p)\prod_{q=2}^{n}d_{i_q}(y_q)\\
&+\sum_{k=1}^{n}\sum_{\mathbf i,\mathbf j\in N_{n+m}^n}\pi^{\mathbf j}\pi^{\mathbf i}d_{j_k}\bigl(d_{i_1}(x_k)\bigr)\prod_{\substack{1\leq p\leq n\\p\ne k}}d_{j_p}(x_p)\prod_{q=2}^{n}d_{i_q}(y_q)\\
&+\sum_{k=1}^{n}\sum_{t=2}^{n}\sum_{\mathbf i,\mathbf j\in N_{n+m}^n}\pi^{\mathbf j}\pi^{\mathbf i}d_{i_1}(x_k)\prod_{\substack{1\leq p\leq n\\p\ne k}}d_{j_p}(x_p)d_{j_k}\bigl(d_{i_t}(y_t)\bigr)\prod_{\substack{2\leq q\leq n\\q\ne t}}d_{i_q}(y_q).
\end{aligned}
\end{equation}
The three sums correspond, respectively, to derivatives of the coefficients, second derivatives of the $x_k$, and second derivatives of the $y_t$.

Thus, we obtain $\mathrm{LHS}-\mathrm{RHS}=S_1+S_2$, where $S_1$ collects the terms involving derivatives of the coefficients $\pi$ and $S_2$ collects terms involving second derivatives of $x_k$ and $y_k$.

Fix $t\in\{1,\ldots,n\}$. In \eqref{eq3.7}, the contribution containing a second derivative of $x_t$ is
$$
\begin{aligned}
&\sum_{\mathbf i,\mathbf j\in N_{n+m}^n}
\pi^{\mathbf i}\pi^{\mathbf j}
\prod_{q=2}^{n}d_{i_q}(y_q)
\prod_{\substack{1\leq p\leq n\\p\ne t}}d_{j_p}(x_p)
d_{i_1}\bigl(d_{j_t}(x_t)\bigr).
\end{aligned}
$$
On the right-hand side, a second derivative of $x_t$ can occur only
in the second sum of \eqref{eq:FI-right-full-expansion} with $k=t$. Therefore
\[
\begin{aligned}
S_2(x_t)
={}&
\sum_{\mathbf i,\mathbf j\in N_{n+m}^n}
\pi^{\mathbf i}\pi^{\mathbf j}
\prod_{q=2}^{n}d_{i_q}(y_q)
\prod_{\substack{1\leq p\leq n\\p\ne t}}d_{j_p}(x_p)
\Bigl(
d_{i_1}d_{j_t}-d_{j_t}d_{i_1}
\Bigr)(x_t).
\end{aligned}
\]
Since the derivations commute pairwise, the operator in parentheses is zero. It follows that
\[
S_2(x_t)=0
\qquad\text{for every }t\in\{1,\ldots,n\}.
\]

Now fix $t\in\{2,\ldots,n\}$. The terms containing a second derivative of $y_t$ occur only in the third sum of \eqref{eq:FI-right-full-expansion}. Hence
$$S_2(y_t)=-\sum_{k=1}^{n}\sum_{\mathbf i,\mathbf j\in N_{n+m}^n}\pi^{\mathbf j}\pi^{\mathbf i}d_{i_1}(x_k)\prod_{\substack{1\leq p\leq n\\p\ne k}}d_{j_p}(x_p)\prod_{\substack{2\leq q\leq n\\q\ne t}}d_{i_q}(y_q)d_{j_k}\bigl(d_{i_t}(y_t)\bigr).$$

For each fixed $k$, reindex the sum by interchanging $i_1$ and $j_k$, namely,
$$(i_1,i_2,\ldots,i_n;\,j_1,\ldots,j_k,\ldots,j_n)\longmapsto(j_k,i_2,\ldots,i_n;\,j_1,\ldots,i_1,\ldots,j_n).$$
This is a bijection of $N_{n+m}^n\times N_{n+m}^n$. After reindexing and renaming the indices, we obtain
$$S_2(y_t)=-\sum_{k=1}^{n}\sum_{\mathbf i,\mathbf j\in N_{n+m}^n}\pi^{(j_1,\ldots,j_{k-1},i_1,j_{k+1},\ldots,j_n)}\pi^{(j_k,i_2,\ldots,i_n)} 
\prod_{p=1}^{n}d_{j_p}(x_p)
\prod_{\substack{2\leq q\leq n\\q\ne t}}d_{i_q}(y_q)
d_{i_1}\bigl(d_{i_t}(y_t)\bigr).
$$ 

Let us reindex the sum by $\mathbf i\mapsto\tau_t(\mathbf i)=(i_t,i_2,\ldots,i_{t-1},i_1,i_{t+1},\ldots,i_n).$
Since $\tau_t$ is a bijection and the derivations commute, we also have
\[
\begin{aligned}
S_2(y_t)=-\sum_{k=1}^{n}\sum_{\mathbf i,\mathbf j\in N_{n+m}^n}&\pi^{(j_1,\ldots,j_{k-1},i_t,j_{k+1},\ldots,j_n)}\pi^{(j_k,i_2,\ldots,i_{t-1},i_1,i_{t+1},\ldots,i_n)}\\
&\times\prod_{p=1}^{n}d_{j_p}(x_p)\prod_{\substack{2\leq q\leq n\\q\ne t}}d_{i_q}(y_q)d_{i_1}\bigl(d_{i_t}(y_t)\bigr).
\end{aligned}
\]
Averaging this expression with the preceding one gives
$$\begin{aligned}
S_2(y_t)={}&-\frac{1}{2}\sum_{\mathbf i,\mathbf j\in N_{n+m}^n}\Bigg[\sum_{k=1}^{n}\pi^{(j_1,\ldots,j_{k-1},i_1,j_{k+1},\ldots,j_n)}\pi^{(j_k,i_2,\ldots,i_n)}\\
&\qquad\qquad+\sum_{k=1}^{n}\pi^{(j_1,\ldots,j_{k-1},i_t,j_{k+1},\ldots,j_n)}\pi^{(j_k,i_2,\ldots,i_{t-1},i_1,i_{t+1},\ldots,i_n)}\Bigg]\\
&\qquad\times\prod_{p=1}^{n}d_{j_p}(x_p)\prod_{\substack{2\leq q\leq n\\q\ne t}}d_{i_q}(y_q)d_{i_1}\bigl(d_{i_t}(y_t)\bigr).
\end{aligned}$$
By Corollary~\ref{s_2(y_t)}, the expression in brackets is zero. Hence $S_2(y_t)=0.$

Consequently, $$S_2=\sum_{t=1}^{n}S_2(x_t)+\sum_{t=2}^{n}S_2(y_t)
=0.$$

It remains to compute $S_1$. The terms arising from differentiating the coefficient $\pi^{\mathbf j}$ on the left-hand side are
\begin{equation}\label{eq:S1-left-coefficient}
\sum_{\mathbf i,\mathbf j\in N_{n+m}^n}
\pi^{\mathbf i}d_{i_1}\bigl(\pi^{\mathbf j}\bigr)
\prod_{q=2}^{n}d_{i_q}(y_q)
\prod_{p=1}^{n}d_{j_p}(x_p).
\end{equation}
For fixed $k$, the corresponding contribution from the right-hand
side, before it is subtracted, is
\begin{equation}\label{eq:S1-right-coefficient}
\begin{aligned}
\sum_{\mathbf i,\mathbf j\in N_{n+m}^n}
&\pi^{\mathbf j}d_{j_k}\bigl(\pi^{\mathbf i}\bigr)
d_{i_1}(x_k)
\prod_{\substack{1\leq p\leq n\\p\ne k}}d_{j_p}(x_p)
\prod_{q=2}^{n}d_{i_q}(y_q).
\end{aligned}
\end{equation}
To put \eqref{eq:S1-right-coefficient} in the same differential order as \eqref{eq:S1-left-coefficient}, interchange $i_1$ and $j_k$ as before, we obtain that \eqref{eq:S1-right-coefficient} becomes
 
\[
\begin{aligned}
\sum_{\mathbf i,\mathbf j\in N_{n+m}^n}
&\pi^{(j_1,\ldots,j_{k-1},i_1,j_{k+1},\ldots,j_n)}
d_{i_1}\bigl(\pi^{(j_k,i_2,\ldots,i_n)}\bigr)
\prod_{q=2}^{n}d_{i_q}(y_q)
\prod_{p=1}^{n}d_{j_p}(x_p).
\end{aligned}
\]
After summing over $k$ and remembering that all right-hand-side terms occur with a minus sign in $\mathrm{LHS}-\mathrm{RHS}$, we obtain
\[
\begin{aligned}
S_1=
\sum_{\mathbf i,\mathbf j\in N_{n+m}^n}
\Bigg(
\pi^{\mathbf i}d_{i_1}\bigl(\pi^{\mathbf j}\bigr)-
\sum_{k=1}^{n}
\pi^{(j_1,\ldots,j_{k-1},i_1,j_{k+1},\ldots,j_n)}
d_{i_1}\bigl(\pi^{(j_k,i_2,\ldots,i_n)}\bigr)
\Bigg)\prod_{q=2}^{n}d_{i_q}(y_q)
\prod_{p=1}^{n}d_{j_p}(x_p).
\end{aligned}
\]
This is exactly the left-hand side of Theorem~\ref{thm3.2}. Since $S_2=0$, the defect of the
fundamental identity is equal to this expression. It follows that
\eqref{eq3.6} holds for all
$x_1,\ldots,x_n,y_2,\ldots,y_n\in\mathcal A$ if and only if
equality in Theorem~\ref{thm3.2} holds for all such elements. The
alternation and Leibniz properties were already established, so this is equivalent that
$(\mathcal A,\cdot,[-,\ldots,-]_A)$ is a Poisson $n$-Lie algebra.
\end{proof}

The following additional assumption allows the differential products in Theorem~\ref{thm3.2} to be isolated, yielding a necessary and sufficient coefficient criterion.

\begin{equation*}
\sum_{\mathbf i,\mathbf j\in N_{n+m}^n}\Bigg(\pi^{\mathbf i}d_{i_1}\bigl(\pi^{\mathbf j}\bigr)-\sum_{k=1}^{n}\pi^{(j_1,\ldots,j_{k-1},i_1,j_{k+1},\ldots,j_n)}d_{i_1}\bigl(\pi^{(j_k,i_2,\ldots,i_n)}\bigr)\Bigg)\prod_{q=2}^{n}d_{i_q}(y_q)\prod_{p=1}^{n}d_{j_p}(x_p)=0
\end{equation*}

\textbf{Assumption.}
For every
$$(u_2,\ldots,u_n)\in N_{n+m}^{n-1},\qquad (v_1,\ldots,v_n)\in N_{n+m}^{n},$$
there exist $x_1,\ldots,x_n,y_2,\ldots,y_n\in\mathcal A$ such that
$$P_{u,v}:=\prod_{q=2}^{n}d_{u_q}(y_q)\prod_{p=1}^{n}d_{v_p}(x_p)$$
is a nonzero element of $\mathcal A$ satisfying 
$$zP_{u,v}=0\Longrightarrow z=0\qquad\text{for every }z\in\mathcal A,$$
while
$$\prod_{q=2}^{n}d_{u'_q}(y_q)\prod_{p=1}^{n}d_{v'_p}(x_p)=0$$
for every pair $(u'_2,\ldots,u'_n)\in N_{n+m}^{n-1}$ and $(v'_1,\ldots,v'_n)\in N_{n+m}^{n}$ distinct from the prescribed pair.

\begin{cor}[Coefficient criterion]
\label{crite}
Under the above assumption, the bracket in
\eqref{eq:determinant-bracket} defines a Poisson $n$-Lie algebra if
and only if
\begin{equation}\label{eq5.5.1}
\begin{aligned}\sum_{a=1}^{n+m}\Bigg(\pi^{(a,i_2,\ldots,i_n)}d_a\bigl(\pi^{(j_1,\ldots,j_n)}\bigr)-\sum_{k=1}^{n}\pi^{(j_1,\ldots,j_{k-1},a,j_{k+1},\ldots,j_n)}d_a\bigl(\pi^{(j_k,i_2,\ldots,i_n)}\bigr)\Bigg)=0
\end{aligned}
\end{equation}
for all $(i_2,\ldots,i_n)\in N_{n+m}^{n-1}$ and
$(j_1,\ldots,j_n)\in N_{n+m}^n$.
\end{cor}

\begin{proof}
The sufficiency follows immediately from Theorem~\ref{thm3.2}, by grouping the terms according to $(i_2,\ldots,i_n)$ and $(j_1,\ldots,j_n)$.

Conversely, suppose that the bracket defines a Poisson $n$-Lie algebra. Fix $(u_2,\ldots,u_n)\in N_{n+m}^{n-1}$ and $(v_1,\ldots,v_n)\in N_{n+m}^n$, and choose $x_1,\ldots,x_n,y_2,\ldots,y_n$ as in the assumption. In the identity of Theorem~\ref{thm3.2}, the differential product vanishes unless
$$(i_2,\ldots,i_n)=(u_2,\ldots,u_n),\qquad (j_1,\ldots,j_n)=(v_1,\ldots,v_n).$$
Summing the remaining terms over $i_1=a$ therefore gives
$$\left(\sum_{a=1}^{n+m}\Bigg(\pi^{(a,u_2,\ldots,u_n)}d_a\bigl(\pi^{(v_1,\ldots,v_n)}\bigr)-\sum_{k=1}^{n}\pi^{(v_1,\ldots,v_{k-1},a,v_{k+1},\ldots,v_n)}d_a\bigl(\pi^{(v_k,u_2,\ldots,u_n)}\bigr)\Bigg)\right)P_{u,v}=0.$$
Since multiplication by $P_{u,v}$ is injective, the coefficient in parentheses is zero. As the prescribed pair was arbitrary, this proves \eqref{eq5.5.1}.
\end{proof}

Returning to the general setting of Theorem~\ref{thm3.2}, we obtain a sufficient condition by requiring each summand in its criterion to vanish separately.

\begin{cor}[Pointwise sufficient criterion]\label{cor:pointwise-sufficient}
The bracket in \eqref{eq:determinant-bracket} defines a Poisson $n$-Lie algebra whenever
\begin{equation}\label{eq:pointwise-differential}
\begin{aligned}
\pi^{\mathbf i}d_{i_1}\bigl(\pi^{\mathbf j}\bigr)=\sum_{k=1}^{n}\pi^{(j_1,\ldots,j_{k-1},i_1,j_{k+1},\ldots,j_n)}d_{i_1}\bigl(\pi^{(j_k,i_2,\ldots,i_n)}\bigr)
\end{aligned}
\end{equation}
for all $\mathbf i,\mathbf j\in N_{n+m}^n$.
\end{cor}

\begin{proof}
For each fixed pair $\mathbf i,\mathbf j$, identity
\eqref{eq:pointwise-differential} implies that the corresponding expression in Theorem~\ref{thm3.2} vanishes. Hence every summand in the criterion vanishes separately.
\end{proof}

\begin{cor}\label{cor:constant-matrix}
Suppose that
$$d_r(a_{p,s})=0\qquad(1\leq r,p\leq n+m,\ 1\leq s\leq m).$$
Then \eqref{eq:determinant-bracket} defines a Poisson $n$-Lie bracket. In particular, this holds when $\mathcal A$ is unital and
$a_{p,s}\in\mathbb F1_{\mathcal A}$ for all $p,s$.
\end{cor}

\begin{proof}
Each complementary minor is a signed sum of products of entries of $A$. The Leibniz rule and the hypothesis therefore give
$$d_r\bigl(\pi^{\mathbf i}\bigr)=0\qquad(r\in N_{n+m},\ \mathbf i\in N_{n+m}^n).$$
Thus \eqref{eq:pointwise-differential} holds with both sides equal to zero. For the last assertion, the Leibniz rule gives $d_r(1_{\mathcal A})=0$. Hence $d_r(\lambda1_{\mathcal A})=\lambda d_r(1_{\mathcal A})=0$ for every $\lambda\in\mathbb F$.
\end{proof}

The assumption of Corollary~\ref{crite} is satisfied by both
$$\mathbb F[t_1,\ldots,t_{n+m}]\qquad\text{and}\qquad\mathbb F[t_1^{\pm1},\ldots,t_{n+m}^{\pm1}]$$
with the Euler derivations
$$d_r=t_r\frac{\partial}{\partial t_r},\qquad 1\leq r\leq n+m.$$
Indeed,
\begin{equation}\label{eq:euler-coordinate}
d_r(t_s)=\delta_{r,s}t_s.
\end{equation}
For prescribed ordered tuples $(u_2,\ldots,u_n)$ and $(v_1,\ldots,v_n)$, take
$$y_q=t_{u_q}\quad (2\leq q\leq n),\qquad x_p=t_{v_p}\quad (1\leq p\leq n).$$
By \eqref{eq:euler-coordinate}, the selected product is
$$\prod_{q=2}^{n}t_{u_q}\prod_{p=1}^{n}t_{v_p}.$$
It is a nonzero monomial and hence a regular element of either integral domain. A different pair of ordered tuples has at least one entry in a different position. Its differential product then contains a factor $d_r(t_s)$ with $r\ne s$, which is zero by \eqref{eq:euler-coordinate}. This verifies the assumption.

\begin{example}\label{ex:determinantal-examples}
The following choices illustrate the construction.
\begin{itemize}

\item[(i)]
Let $\mathcal A$ be a commutative associative algebra with mutually commuting derivations $d_1,\ldots,$ $d_{n+m}$. For fixed elements $y_j$, $1\leq j\leq m$, of $\mathcal A$, set $a_{i,j}=d_i(y_j)$. Then, by \cite{Filippov2}, $(\mathcal A,\cdot,[-,\ldots,-])$ is a Poisson $n$-Lie algebra.

\item[(ii)] Let $\mathcal A$ be a unital commutative associative algebra. Suppose that one column of $A$ is arbitrary, while each of the remaining columns has exactly one nonzero scalar entry, with these entries lying in pairwise distinct rows. Expanding the determinant along these scalar columns reduces the construction to the case in which $A$ has a single column. Hence the examples in \cite{newexample} yield Poisson $n$-Lie algebras of the present form.

\item[(iii)]
Suppose that $\mathcal A$ is unital, take $f\in\mathcal A$, and let $B\in M_{n+m,m}(\mathbb F)$. Set $A=fB$, and write $\pi_B^{\mathbf i}$ for the alternating complementary coordinate of $B$ corresponding to $\mathbf i$. Since every column is multiplied by $f$, multilinearity of the determinant gives $\pi^{\mathbf i}=f^m\pi_B^{\mathbf i}.$ Since every $\pi_B^{\mathbf i}$ is a scalar, one has $d_r\bigl(\pi^{\mathbf i}\bigr)=d_r(f^m)\pi_B^{\mathbf i}.$ Therefore the left-hand side of \eqref{eq:pointwise-differential} is $f^m d_{i_1}(f^m)\pi_B^{\mathbf i}\pi_B^{\mathbf j},$ while the right-hand side is
$$f^m d_{i_1}(f^m)\sum_{k=1}^{n}\pi_B^{(j_1,\ldots,j_{k-1},i_1,j_{k+1},\ldots,j_n)}\pi_B^{(j_k,i_2,\ldots,i_n)}=f^m d_{i_1}(f^m)\pi_B^{\mathbf i}\pi_B^{\mathbf j},$$
where the equality follows from \eqref{eq:plucker-exchange} applied to $B$. Thus  Corollary~\ref{cor:pointwise-sufficient} applies for every $f$ and every constant matrix $B$.

\item[(iv)]
Let $(\mathcal A,\cdot)$ be a unital commutative associative algebra equipped with mutually commuting derivations $d_1,\ldots,d_{n+m}$, and let $A\in M_{n+m,m}(\mathbb F)$. Then $(\mathcal A,\cdot,[-,\ldots,-])$ is a Poisson $n$-Lie algebra.
\end{itemize}
\end{example}

\section{Constructions between Poisson algebras and Poisson $n$-Lie algebras}

Now we start with a tensor-product construction that produces a Poisson $n$-Lie algebra from a given Poisson $n$-Lie algebra and a commutative associative algebra.

\begin{pro}\label{prop4.1}
Let $(\mathcal P,\cdot,[-,\ldots,-])$ be a Poisson $n$-Lie algebra, and let $(\mathcal B,\circ)$ be a commutative associative algebra. The product $\cdot_{\otimes}$ is associative and commutative since the associative products on $\mathcal P$ and $\mathcal B$ are both associative and commutative. Define operations $\cdot_{\otimes}$ and $[-,\ldots,-]_{\otimes}$ on $\mathcal{P}\otimes\mathcal{B}$ by
$$(x\otimes a)\cdot_{\otimes}(y\otimes b)=(x\cdot y)\otimes(a\circ b)$$
and
$$[x_1\otimes a_1,\ldots,x_n\otimes a_n]_{\otimes}=[x_1,\ldots,x_n]\otimes(a_1\circ\cdots\circ a_n).$$
Then $(\mathcal{P}\otimes\mathcal{B},\cdot_{\otimes},[-,\ldots,-]_{\otimes})$ is a Poisson $n$-Lie algebra.
\end{pro}

\begin{proof}
The skew-symmetry and fundamental identity for $[-,\ldots,-]_{\otimes}$ follow directly from the corresponding properties of the bracket on $\mathcal{P}$ and the associativity  and commutativity of $\mathcal{B}$. Finally, a direct computation verifies the Leibniz rule.
\end{proof}

A useful link between Lie algebras and $n$-Lie algebras is provided by Leibniz $n$-algebras. Recall that a vector space $\mathcal{L}$ over a field $\mathbb{F}$, equipped with an $n$-linear bracket $[-,\ldots,-]\colon\mathcal{L}^{\otimes n}\to\mathcal{L}$, is called a \emph{Leibniz $n$-algebra} if it satisfies the fundamental identity \eqref{eq2.1}, without any skew-symmetry assumption (see~\cite{Casas}). Thus Leibniz $n$-algebras generalize $n$-Lie algebras. For $n=2$, one recovers Leibniz algebras, which in turn generalize Lie algebras (see~\cite{Loday}).

The following result recalls a construction from~\cite{Casas} that associates a Leibniz $n$-algebra with every Leibniz algebra.

\begin{pro}\label{prop4.2}
Let $(\mathcal{L},[-,-])$ be a Leibniz algebra. Then $\mathcal{L}$ becomes a Leibniz $n$-algebra under the $n$-ary bracket
$$[x_1,\ldots,x_n]:=[x_1,[x_2,\ldots,[x_{n-1},x_n]\ldots]],\qquad x_1,\ldots,x_n\in\mathcal{L}.$$
\end{pro}

We next recall the standard Poisson structure on the tensor product of two Poisson algebras (see~\cite{Xu}). Let $(\mathcal{P}_1,\cdot_1,[-,-]_1)$ and $(\mathcal{P}_2,\cdot_2,[-,-]_2)$ be Poisson algebras. Their tensor product $\mathcal{P}=\mathcal{P}_1\otimes\mathcal{P}_2$ is a Poisson algebra with the componentwise associative product and the bracket
$$[x_1\otimes y_1,x_2\otimes y_2]_{\otimes}=[x_1,x_2]_1\otimes(y_1\cdot_2y_2)+(x_1\cdot_1x_2)\otimes[y_1,y_2]_2.$$

For $x_1,\ldots,x_n\in\mathcal{P}$, define an $n$-ary bracket on $\mathcal{P}$ by
\begin{equation}\label{eq3.1}
[x_1,\ldots,x_n]:=[x_1,[x_2,\ldots,[x_{n-1},x_n]_{\otimes}\ldots]_{\otimes}]_{\otimes}.
\end{equation}
Let $\mathcal{I}$ be the Poisson ideal of $\mathcal{P}$ generated by the elements
$$[x_1,\ldots,x_i,\ldots,x_j,\ldots,x_n]+[x_1,\ldots,x_j,\ldots,x_i,\ldots,x_n],\qquad 1\leq i<j\leq n,$$
where $x_1,\ldots,x_n\in\mathcal{P}$.

\begin{pro}\label{prop4.3}
With the notation above, the bracket
$[\overline{x}_1,\ldots,\overline{x}_n]:=\overline{[x_1,\ldots,x_n]}$
is well defined on $\mathcal{P}/\mathcal{I}$ and makes it into a Poisson $n$-Lie algebra.
\end{pro}

\begin{proof}
Since $(\mathcal{P},[-,-]_{\otimes})$ is a Lie algebra, Proposition~\ref{prop4.2} implies that the bracket defined by~\eqref{eq3.1} satisfies the fundamental identity, although it need not be skew-symmetric. Because $\mathcal{I}$ is a Poisson ideal, the nested bracket descends to $\mathcal{P}/\mathcal{I}$, by the definition of $\mathcal{I}$, the induced $n$-ary bracket is skew-symmetric. Hence it defines an $n$-Lie algebra structure on the quotient.

It remains to verify compatibility with the associative product. Put
$$X=[x_2,[x_3,\ldots,[x_{n-1},x_n]_{\otimes}\ldots]_{\otimes}]_{\otimes}.$$
The Leibniz rule for the binary Poisson bracket gives
\begin{align*}
[yz,x_2,\ldots,x_n]
&=[yz,X]_{\otimes}\\
&=y[z,X]_{\otimes}+z[y,X]_{\otimes}\\
&=y[z,x_2,\ldots,x_n]+z[y,x_2,\ldots,x_n].
\end{align*}
Thus the $n$-ary bracket satisfies the Leibniz rule in its first argument. Since the induced bracket on $\mathcal{P}/\mathcal{I}$ is skew-symmetric, it satisfies the Leibniz rule in every argument. Therefore $\mathcal{P}/\mathcal{I}$ is a Poisson $n$-Lie algebra.
\end{proof}

In particular, taking $\mathcal P_1=\mathcal P$ and $\mathcal P_2=\mathcal P^c$, where $\mathcal P^c$ denotes an isomorphic copy of $\mathcal P$, Proposition~\ref{prop4.3} yields a Poisson $n$-Lie algebra structure on $(\mathcal P\otimes\mathcal P^c)/\mathcal I$.

We now turn to the construction of ordinary Poisson algebras from Poisson $n$-Lie algebras. For this purpose, we recall the tensor construction that associates a Leibniz algebra with an $n$-Lie algebra~\cite{Daletskii}.

Let $(\mathcal{L},[-,\ldots,-])$ be an $n$-Lie algebra. For pure
tensors
$$x=x_1\otimes\cdots\otimes x_{n-1},\qquad y=y_1\otimes\cdots\otimes y_{n-1}$$
in $\mathcal{L}^{\otimes(n-1)}$, define
$$[x,y]_{\otimes}:=\sum_{i=1}^{n-1}y_1\otimes\cdots\otimes[x_1,\ldots,x_{n-1},y_i]\otimes\cdots\otimes y_{n-1},$$
and extend this operation bilinearly to $\mathcal{L}^{\otimes(n-1)}$.

\begin{pro}\label{prop3.4}
The bracket $[-,-]_{\otimes}$ makes $\mathcal{L}^{\otimes(n-1)}$ into a Leibniz algebra.
\end{pro}

We denote the Leibniz algebra in Proposition~\ref{prop3.4} by $\widetilde{\mathcal{L}}$. Define
$$\operatorname{ad}\colon\mathcal{L}^{\otimes(n-1)}\longrightarrow\operatorname{End}(\mathcal{L})$$
by $\operatorname{ad}_{x}(y)=[x_1,\ldots,x_{n-1},y],$ with $x=x_1\otimes\cdots\otimes x_{n-1}$ and $y\in\mathcal{L}.$ Then
$$\ker(\operatorname{ad})=\{x\in\mathcal{L}^{\otimes(n-1)}\mid\operatorname{ad}_{x}(y)=0\text{ for all }y\in\mathcal{L}\}$$
is an ideal of $\widetilde{\mathcal{L}}$. Moreover,
$$\mathcal{J}=\operatorname{span}\{[x,y]_{\otimes}+[y,x]_{\otimes}\mid x,y\in\mathcal{L}^{\otimes(n-1)}\}$$
is an ideal of $\widetilde{\mathcal{L}}$, and $\widetilde{\mathcal{L}}/\mathcal{J}$ is a Lie algebra. 

Let now $\mathcal{P}$ be a Poisson $n$-Lie algebra. The tensor power $\widetilde{\mathcal{P}}=\mathcal{P}^{\otimes(n-1)}$ is a commutative associative algebra under the componentwise product
$$
(x_1\otimes\cdots\otimes x_{n-1})\cdot(y_1\otimes\cdots\otimes y_{n-1})
:=x_1y_1\otimes\cdots\otimes x_{n-1}y_{n-1}.
$$
We next adapt the preceding construction to the Poisson setting.

\begin{pro}
Let $\mathcal{P}$ be a Poisson $n$-Lie algebra, and write  $\widetilde{\mathcal{P}}=\mathcal{P}^{\otimes(n-1)}.$ Let $\mathcal{J}$ be the ideal of $\widetilde{\mathcal{P}}$ with respect to both the associative product $\cdot$ and the Leibniz bracket $[-,-]_{\otimes}$ generated by all elements of the form $[x,y]_{\otimes}+[y,x]_{\otimes},$ $x,y\in\widetilde{\mathcal{P}}.$ Then the operations induced by $\cdot$ and $[-,-]_{\otimes}$ make $\widetilde{\mathcal{P}}/\mathcal{J}$ into a Poisson algebra.
\end{pro}

\begin{proof}
Since $\mathcal{J}$ is an ideal with respect to both operations, the product and the bracket descend to $\widetilde{\mathcal{P}}/\mathcal{J}$. The induced product is commutative and associative. By Proposition~\ref{prop3.4}, the bracket $[-,-]_{\otimes}$ satisfies the Leibniz identity. Moreover, the definition of $\mathcal{J}$ implies that the induced bracket is skew-symmetric. A skew-symmetric Leibniz bracket satisfies the Jacobi identity. Hence the induced bracket makes $\widetilde{\mathcal{P}}/\mathcal{J}$ into a Lie algebra.

It remains to verify the Leibniz rule. By bilinearity, it suffices to consider pure tensors
$$
x=x_1\otimes\cdots\otimes x_{n-1},\qquad
y=y_1\otimes\cdots\otimes y_{n-1},\qquad
z=z_1\otimes\cdots\otimes z_{n-1}.
$$
In the following computation, all equalities are understood in $\widetilde{\mathcal{P}}/\mathcal{J}$. Using the skew-symmetry of the induced bracket and the Leibniz rule in $\mathcal{P}$, we obtain
\begin{align*}
[x\cdot y,z]_{\otimes}
&=-[z,x\cdot y]_{\otimes}\\
&=-\sum_{i=1}^{n-1}x_1y_1\otimes\cdots\otimes[z_1,\ldots,z_{n-1},x_i y_i]\otimes\cdots\otimes x_{n-1}y_{n-1}\\
&=-\sum_{i=1}^{n-1}x_1y_1\otimes\cdots\otimes x_i[z_1,\ldots,z_{n-1},y_i]\otimes\cdots\otimes x_{n-1}y_{n-1}\\
&\quad-\sum_{i=1}^{n-1}x_1y_1\otimes\cdots\otimes y_i[z_1,\ldots,z_{n-1},x_i]\otimes\cdots\otimes x_{n-1}y_{n-1}.
\end{align*}
By the commutativity of the product in $\mathcal{P}$, the last expression can be rewritten as
$$-x\cdot[z,y]_{\otimes}-y\cdot[z,x]_{\otimes}.$$
Applying skew-symmetry once more yields
$$[x\cdot y,z]_{\otimes}=x\cdot[y,z]_{\otimes}+y\cdot[x,z]_{\otimes}.$$
Thus the induced bracket satisfies the Leibniz rule. Therefore, $\widetilde{\mathcal{P}}/\mathcal{J}$ is a Poisson algebra.
\end{proof}

\section{Solvable Poisson $n$-Lie Algebras}

In this section, we investigate solvability and nilpotency for Poisson $n$-Lie algebras. We establish analogues of Engel's theorem and Lie's theorem, characterize solvability in terms of the square ideal, and study solvable extensions determined by maximal hypo-nilpotent ideals. The latter play, for $n>2$, a role analogous to that of the nilradical in the structure theory of solvable Poisson algebras.

Let $\mathcal P=(\mathcal P,\cdot,[-,\ldots,-])$ be a Poisson $n$-Lie algebra. We write
$$
\mathcal P_A=(\mathcal P,\cdot)
\qquad\text{and}\qquad
\mathcal P_L=(\mathcal P,[-,\ldots,-])
$$
for its underlying commutative associative algebra and its underlying $n$-Lie algebra, respectively. For $x\in\mathcal P$, let $P_x\in\operatorname{End}(\mathcal P)$ denote left multiplication by $x$. For
$X=x_1\wedge\cdots\wedge x_{n-1}\in\bigwedge^{n-1}\mathcal P$, set
$$
Q_X(z)=[x_1,\ldots,x_{n-1},z],
\qquad z\in\mathcal P.
$$
Thus $Q_X$ is the inner derivation associated with $X$.

A subspace of $\mathcal P$ is a subalgebra, respectively an ideal, of the Poisson $n$-Lie algebra if it is simultaneously a subalgebra, respectively an ideal, of both $\mathcal P_A$ and $\mathcal P_L$.

Let $\mathcal I$ be an ideal of $\mathcal P$. Its derived series and lower central series are defined by
\begin{alignat}{2}
\mathcal I^{(1)}&=\mathcal I,
&\qquad
\mathcal I^{(k+1)}
&=[\mathcal I^{(k)},\mathcal I^{(k)},\mathcal P,\ldots,\mathcal P]
+\mathcal I^{(k)}\cdot\mathcal I^{(k)},
\label{eq5.1}\\
\mathcal I^1&=\mathcal I,
&
\mathcal I^{k+1}
&=[\mathcal I^k,\mathcal I,\mathcal P,\ldots,\mathcal P]
+\mathcal I^k\cdot\mathcal I,
\qquad k\geq1.
\label{eq5.2}
\end{alignat}
Here and below, the unfilled entries of an $n$-fold bracket are occupied by copies of $\mathcal P$. Each term in these two series is an ideal of $\mathcal P$.

\begin{defn}\label{defn5.3}
An ideal $\mathcal I$ of $\mathcal P$ is called \emph{solvable} if $\mathcal I^{(s)}=0$ for some $s\geq1$, and \emph{nilpotent} if $\mathcal I^s=0$ for some $s\geq1$. The least such $s$ is called the \emph{solvability index}, respectively the \emph{nilpotency index}, of $\mathcal I$.
\end{defn}

If $\mathcal P$ is solvable, then $\mathcal P_L$ is solvable and $\mathcal P_A$ is nilpotent. Solvability and nilpotency are inherited by subalgebras and quotient algebras.

We first record several elementary relations between the two series.

\begin{lem}\label{lem5.1}
Let $\mathcal P$ be a Poisson $n$-Lie algebra. For $i,j\geq1$ and
$i_1,\ldots,i_k\geq1$, where $1\leq k\leq n$, one has
\begin{equation}\label{eq5.3}
\begin{aligned}
\mathcal P^i\cdot\mathcal P^j
&\subseteq\mathcal P^{i+j},\\
[\mathcal P^{i_1},\ldots,\mathcal P^{i_k},
  \mathcal P,\ldots,\mathcal P]
&\subseteq
\mathcal P^{i_1+\cdots+i_k-k+2},\\
\mathcal P^{(i)}
&\subseteq\mathcal P^{2^{i-1}},\\
(\mathcal P^{(i)})^{(j)}
&=\mathcal P^{(i+j-1)}.
\end{aligned}
\end{equation}
\end{lem}

\begin{proof}
We prove the first inclusion by induction on $j$. The case $j=1$ follows from \eqref{eq5.2}. Assuming the assertion for $j$, the Leibniz identity gives
\begin{align*}
\mathcal P^i\mathcal P^{j+1}&=\mathcal P^i\bigl([\mathcal P^j,\mathcal P,\ldots,\mathcal P]      +\mathcal P^j\mathcal P\bigr)\\
&\overset{\eqref{eq2.2}}{\subseteq}[\mathcal P^i\mathcal P^j,\mathcal P,\ldots,\mathcal P]+\mathcal P^j[\mathcal P^i,\mathcal P,\ldots,\mathcal P]+(\mathcal P^i\mathcal P^j)\mathcal P\\[2mm]
&\subseteq[\mathcal P^{i+j},\mathcal P,\ldots,\mathcal P]+\mathcal P^j\mathcal P^{i+1}+\mathcal P^{i+j}\mathcal P\subseteq\mathcal P^{i+j+1}.
\end{align*}

The second inclusion follows by induction on the number $k$ of nontrivial entries, combined with an induction on the last exponent. The case $k=1$ is precisely \eqref{eq5.2}. In the induction step, write
$$
\mathcal P^{i_{k+1}+1}
=[\mathcal P^{i_{k+1}},\mathcal P,\ldots,\mathcal P]
+\mathcal P^{i_{k+1}}\mathcal P.
$$
The Filippov identity applied to the first summand and the Leibniz identity applied to the second reduce the required inclusion to the induction hypotheses and the first inclusion already proved. This yields
\begin{equation}\label{eq5.4}
[\mathcal P^{i_1},\ldots,\mathcal P^{i_{k+1}},
  \mathcal P,\ldots,\mathcal P]
\subseteq
\mathcal P^{i_1+\cdots+i_{k+1}-(k+1)+2}.
\end{equation}

For the third inclusion, assume inductively that
$\mathcal P^{(i)}\subseteq\mathcal P^{2^{i-1}}$. The first two inclusions imply
\begin{align*}
\mathcal P^{(i+1)}
&=[\mathcal P^{(i)},\mathcal P^{(i)},
   \mathcal P,\ldots,\mathcal P]
  +\mathcal P^{(i)}\mathcal P^{(i)}\subseteq\mathcal P^{2^i}.
\end{align*}
The final equality follows directly by iterating the definition of the derived series.
\end{proof}

In particular, every nilpotent Poisson $n$-Lie algebra is solvable. We shall also use the standard fact that the sum of two solvable ideals is solvable and that the sum of two nilpotent ideals is nilpotent. For the latter assertion, a mixed-term induction gives
$$
(\mathcal I+\mathcal J)^r
\subseteq
\mathcal I^{\lfloor r/2\rfloor}
+\mathcal J^{\lfloor r/2\rfloor},
\qquad r\geq2,
$$
and the corresponding argument for the derived series proves the solvable case.

To compare nilpotency in $\mathcal P$ with nilpotency in its two underlying algebras, define
\begin{align*}
\mathcal I_A^1=\mathcal I_L^1=\mathcal I, \qquad \mathcal I_A^{r+1}=\mathcal I_A^r\cdot\mathcal I, \qquad \mathcal I_L^{r+1}=[\mathcal I_L^r,\mathcal I, \mathcal P,\ldots,\mathcal P].\end{align*}
Thus $\mathcal I_A^r$ is the $r$th associative power of $\mathcal I$, whereas $\mathcal I_L^r$ is the $r$th term of its lower central series as an ideal of $\mathcal P_L$.

\begin{lem}\label{lem5.2}
For every $k\geq1$,
\begin{equation}\label{eq5.5}
\mathcal I^k
\subseteq
\sum_{t=0}^{k}
\ \sum_{\substack{r_0\geq0,\ r_1,\ldots,r_t\geq1\\
                   r_0+r_1+\cdots+r_t=k}}
\mathcal I_A^{r_0}
\cdot\prod_{j=1}^{t}\mathcal I_L^{r_j}.
\end{equation}
When $r_0=0$, the factor $\mathcal I_A^{r_0}$ is omitted; when $t=0$, the product over $j$ is omitted. Thus the $t=0$ summand is $\mathcal I_A^k$.
\end{lem}

\begin{proof}
Repeated use of the Leibniz identity gives the following two inclusions:
\begin{equation}\label{eq5.6}
\left\{
\begin{array}{lllll}
\relax[\mathcal I_A^r,\mathcal I,\mathcal P,\ldots,\mathcal P]
&\subseteq
\mathcal I_A^{r-1}\mathcal I_L^2,
\qquad r\geq1,\\[3mm]
\left[
\prod_{j=1}^{t}\mathcal I_L^{r_j},
\mathcal I,\mathcal P,\ldots,\mathcal P
\right]
&\subseteq
\sum_{j=1}^{t}
\left(
\prod_{\substack{\ell=1\\ \ell\neq j}}^{t}
\mathcal I_L^{r_\ell}
\right)
\mathcal I_L^{r_j+1}.
\end{array}
\right.
\end{equation}
In the first line, the factor $\mathcal I_A^0$ is again omitted.

We now argue by induction on $k$. The assertion is immediate for $k=1$. Suppose it holds for $k$. Since
$$
\mathcal I^{k+1}
=[\mathcal I^k,\mathcal I,\mathcal P,\ldots,\mathcal P]
+\mathcal I^k\mathcal I,
$$
the second summand increases $r_0$ by one. By \eqref{eq5.6}, the first summand either replaces $r_0$ by $r_0-1$ and introduces a factor $\mathcal I_L^2$, or replaces one $r_j$ by $r_j+1$. In every case, the total weight increases from $k$ to $k+1$, which proves \eqref{eq5.5}.
\end{proof}

Applied to $\mathcal I=\mathcal P$, the preceding estimate yields the following characterization of nilpotency.

\begin{pro}\label{prop5.3}
A Poisson $n$-Lie algebra $\mathcal P$ is nilpotent if and only if both $\mathcal P_A$ and $\mathcal P_L$ are nilpotent.
\end{pro}

\begin{proof}
The forward implication is immediate. Conversely, suppose that
$\mathcal P_A^a=0$ and $\mathcal P_L^b=0$. In every nonzero summand on the right-hand side of \eqref{eq5.5}, one must have $r_0<a$, $r_j<b$, and $r_0+t<a$, because the displayed product is an associative product of $r_0+t$ factors from $\mathcal P$. These inequalities bound the total weight $r_0+r_1+\cdots+r_t$. Hence $\mathcal P^k=0$ for all sufficiently large $k$.
\end{proof}
We next obtain an Engel-type criterion.

\begin{thm}[Engel criterion]
Let $\mathcal P$ be a finite-dimensional Poisson $n$-Lie algebra over a field of characteristic zero. Then $\mathcal P$ is nilpotent if and only if $P_x$ and $Q_X$ are nilpotent for every $x\in\mathcal P$ and every $X\in\bigwedge^{n-1}\mathcal P$.
\end{thm}

\begin{proof}
If $\mathcal P$ is nilpotent, then Proposition~\ref{prop5.3} shows that both $\mathcal P_A$ and $\mathcal P_L$ are nilpotent. It follows that all operators $P_x$ and $Q_X$ are nilpotent.

Conversely, suppose that all these operators are nilpotent. Since $\mathcal P_A$ is commutative and associative, the multiplication
operators $P_x$, $x\in\mathcal P$, commute. By assumption, they can be simultaneously represented by strictly
upper triangular matrices. Writing $d=\dim\mathcal P$, we obtain
\[
x_1\cdots x_{d+1}
=P_{x_1}\cdots P_{x_d}(x_{d+1})=0
\]
for all $x_1,\ldots,x_{d+1}\in\mathcal P$. Thus
$\mathcal P_A^{d+1}=0$, and hence $\mathcal P_A$ is nilpotent. By the Engel theorem for $n$-Lie algebras \cite{Kasymov}, the nilpotency of every $Q_X$ implies that $\mathcal P_L$ is nilpotent. Proposition~\ref{prop5.3} now yields the nilpotency of $\mathcal P$.
\end{proof}

For the Lie-type result, we use the following analogues of Lie's theorem for $n$-Lie algebras.

\begin{thm}\cite{Kasymov}\label{thm5.4}
Let $\mathcal L$ be a finite-dimensional solvable $n$-Lie algebra over an algebraically closed field of characteristic zero and $V$ be a finite-dimensional $\mathcal L$-module. Then the representation $\rho$ of $\mathcal L$ in the space $V$ is triangularizable, i.e.,
$V$ contains an $\mathcal L$-invariant flag of subspaces.
\end{thm}

For a solvable Poisson $n$-Lie algebra, we combine this triangularization result with the nilpotency of the associative algebra to obtain a common eigenvector for the multiplication operators and the inner derivations.

\begin{thm}[Lie criterion]\label{thm5.5}
Let $\mathcal P$ be a finite-dimensional solvable Poisson $n$-Lie algebra over an algebraically closed field of characteristic zero. Then the family
$$\{P_x\mid x\in\mathcal P\}\cup\{Q_X\mid X\in\textstyle\bigwedge^{n-1}\mathcal P\}$$
has a common eigenvector.
\end{thm}

\begin{proof}
Set $m=\dim\mathcal P$. Since $\mathcal P$ is solvable, $\mathcal P_A$ is nilpotent and $\mathcal P_L$ is solvable. Applying Theorem~\ref{thm5.4} to the regular representation of $\mathcal P_L$, we obtain a complete flag of ideals
\begin{equation}\label{eq5.8}
0=\mathcal L_0\subset\mathcal L_1\subset\cdots\subset\mathcal L_m=\mathcal P_L,\qquad \dim\mathcal L_i=i.
\end{equation}

Let
$$\mathcal C=\operatorname{Ann}(\mathcal P_A)=\{c\in\mathcal P\mid c\mathcal P=0\}.$$
The nilpotency of $\mathcal P_A$ implies that $\mathcal C\neq0$. Moreover, the Leibniz identity shows that $\mathcal C$ is an ideal of $\mathcal P_L$, and hence an ideal of $\mathcal P$. Choose the least $i_0$ for which $\mathcal L_{i_0}\cap\mathcal C\neq0$. Since $\dim(\mathcal L_{i_0}/\mathcal L_{i_0-1})=1$, the ideal $\mathcal V=\mathcal L_{i_0}\cap\mathcal C$ is one-dimensional. Therefore, if $0\neq v\in\mathcal V$, then $P_x(v)=0$ and $Q_X(v)=\lambda_Xv$ for all $x\in\mathcal P$ and $X\in\bigwedge^{n-1}\mathcal P$.
\end{proof}

\begin{cor}
Let $\mathcal P$ be as in Theorem~\ref{thm5.5}. Then $\mathcal P$ admits a complete flag of ideals
$$0=\mathcal P_0\subset\mathcal P_1\subset\cdots\subset\mathcal P_m=\mathcal P,\qquad \dim\mathcal P_i=i.$$
\end{cor}

\begin{proof}
The common eigenline supplied by Theorem~\ref{thm5.5} is a one-dimensional ideal of $\mathcal P$. The quotient by this ideal is again solvable. The result follows by induction on $\dim\mathcal P$.
\end{proof}

We now characterize solvability in terms of the square ideal
$$\mathcal P^2=[\mathcal P,\ldots,\mathcal P]+\mathcal P\cdot\mathcal P.$$

\begin{pro}\label{prop5.8}
Let $\mathcal P$ be a finite-dimensional Poisson $n$-Lie algebra over an algebraically closed field of characteristic zero. Then $\mathcal P$ is solvable if and only if $\mathcal P^2$ is a nilpotent ideal of $\mathcal P$.
\end{pro}

\begin{proof}
Suppose first that $\mathcal P^2$ is nilpotent. Choose $k$ such that
$(\mathcal P^2)^{2^{k-1}}=0$. Lemma~\ref{lem5.1} gives
$$\mathcal P^{(k+1)}=(\mathcal P^{(2)})^{(k)}=(\mathcal P^2)^{(k)}\subseteq(\mathcal P^2)^{2^{k-1}}=0,$$
so $\mathcal P$ is solvable.

Conversely, assume that $\mathcal P$ is solvable. Then $\mathcal P_A$ is nilpotent and $\mathcal P_L$ is solvable. By \cite[Theorem~6]{Kasymov}, the $n$-Lie ideal $[\mathcal P,\ldots,\mathcal P]$ is nilpotent. Put $\mathcal B=\mathcal P\cdot\mathcal P$ and denote by $\mathcal B_L^k$ its lower central series as an ideal of $\mathcal P_L$. An induction based on the Leibniz identity yields
\begin{equation}\label{eq5.9}
\mathcal B_L^k\subseteq\mathcal P_A^{k+1}, \qquad k\geq1.
\end{equation}
Thus $\mathcal B$ is also a nilpotent ideal of $\mathcal P_L$. Since the sum of nilpotent ideals of an $n$-Lie algebra is nilpotent,
$$\mathcal P^2=[\mathcal P,\ldots,\mathcal P]+\mathcal B$$
is nilpotent as an ideal of $\mathcal P_L$. It is also nilpotent as an associative ideal because $\mathcal P_A$ is nilpotent. Lemma~\ref{lem5.2}, applied to $\mathcal I=\mathcal P^2$, now shows that $\mathcal P^2$ is nilpotent as an ideal of $\mathcal P$.
\end{proof}

\begin{cor}
Let $\mathcal P$ be a finite-dimensional Poisson $n$-Lie algebra over an algebraically closed field of characteristic zero. Then $\mathcal P$ is solvable if and only if $\mathcal P_L$ is solvable and $\mathcal P_A$ is nilpotent.
\end{cor}

\begin{proof}
Only the converse requires proof. If $\mathcal P_L$ is solvable and $\mathcal P_A$ is nilpotent, the argument in the second half of the proof of Proposition~\ref{prop5.8} shows that $\mathcal P^2$ is nilpotent. Proposition~\ref{prop5.8} then implies that $\mathcal P$ is solvable.
\end{proof}

For a finite-dimensional Poisson $n$-Lie algebra $\mathcal P$, let $\operatorname{Nil}(\mathcal P)$ denote its nilradical, that is, its largest nilpotent ideal.

\begin{cor}\label{cor5.11}
Let $\mathcal P$ be a finite-dimensional solvable Poisson $n$-Lie algebra over an algebraically closed field of characteristic zero. Then $\operatorname{Nil}(\mathcal P)$ coincides with the nilradical of $\mathcal P_L$.
\end{cor}

\begin{proof}
Let $\mathcal N_L$ be the nilradical of $\mathcal P_L$. Clearly, $\operatorname{Nil}(\mathcal P)\subseteq\mathcal N_L$. By Proposition~\ref{prop5.8},
$$\mathcal N_L\cdot\mathcal P\subseteq\mathcal P^2\subseteq\operatorname{Nil}(\mathcal P)\subseteq\mathcal N_L.$$
Hence $\mathcal N_L$ is also an associative ideal. Since $\mathcal P_A$ is nilpotent, $\mathcal N_L$ is nilpotent as an associative algebra. It is nilpotent as an $n$-Lie ideal by definition, and Lemma~\ref{lem5.2} therefore shows that it is a nilpotent ideal of $\mathcal P$. Thus
$\mathcal N_L\subseteq\operatorname{Nil}(\mathcal P)$.
\end{proof}

\begin{cor}
Let $\mathcal P$ be as in Corollary~\ref{cor5.11}. Then $\operatorname{Nil}(\mathcal P)$ is invariant under every derivation of $\mathcal P$.
\end{cor}

\begin{proof}
By Corollary~\ref{cor5.11}, $\operatorname{Nil}(\mathcal P)$ is the nilradical of $\mathcal P_L$. The latter is characteristic by \cite[Theorem~2]{Kasymov}. Every derivation of $\mathcal P$ is, in particular, a derivation of $\mathcal P_L$, and the assertion follows.
\end{proof}

For ordinary Poisson algebras, classical nilpotency criteria involving a nonsingular derivation or a fixed-point-free automorphism of prime order remain valid; see \cite{Amir2,Jacobson}. These criteria do not extend to $n$-Lie algebras when $n>2$ \cite[Example~3.1]{Williams}, and therefore they fail in general for Poisson $n$-Lie algebras as well.

For Lie algebras, an ideal is nilpotent if and only if it is nilpotent as a subalgebra. For $n$-Lie algebras with $n>2$, the two notions differ; see \cite{Abdurasulov,BaiR}. This motivates the following definition.

\begin{defn} An ideal $\mathcal I$ of a Poisson $n$-Lie algebra $\mathcal P$ is called a {\it hypo-nilpotent ideal} if $\mathcal I$ is nilpotent as a subalgebra but not nilpotent as an ideal. A hypo-nilpotent ideal that is not properly contained in any other hypo-nilpotent ideal is called {\it a maximal hypo-nilpotent ideal} of $\mathcal P$.
\end{defn}

For an ideal $\mathcal I$, its lower central series as a subalgebra is
$$\mathcal I^{[1]}=\mathcal I,\qquad\mathcal I^{[k+1]}=[\mathcal I^{[k]},\mathcal I,\ldots,\mathcal I]+\mathcal I^{[k]}\cdot\mathcal I,\qquad k\geq1.$$

The following example shows that hypo-nilpotent ideals need not be closed under sums.

\begin{example}
Let $4\leq n\leq m$, and let $\mathcal P$ have basis $e_1,\ldots,e_{m+1}$. For $1\leq p\leq m-n+1$ and $m-n+2\leq i,j\leq m$, define the nonzero brackets and products by
$$[e_p,e_{m-n+2},\ldots,e_m]=e_p, \qquad e_i e_j=\alpha_{ij}e_{m+1}.$$
Here $\alpha_{ij}=\alpha_{ji}\in\mathbb C$, and all products and brackets not forced by commutativity, skew-symmetry, and multilinearity are zero.

For $r=1,2$, set
$$\mathcal I_r=\operatorname{span}\{e_1,\ldots,\widehat e_{m-r},\ldots,e_{m+1}\}.$$
Both $\mathcal I_1$ and $\mathcal I_2$ are ideals. Because $\mathcal I_r$ omits one of the entries required for a nonzero bracket,
$$\mathcal I_r^{[2]}\subseteq\mathbb C e_{m+1}\qquad\text{and hence}\qquad\mathcal I_r^{[3]}=0.$$
On the other hand, the defining bracket shows that $\mathcal I_r^k\neq0$ for every $k\geq1$. Thus each $\mathcal I_r$ is hypo-nilpotent. Moreover, $\mathcal I_1+\mathcal I_2=\mathcal P$, whereas $\mathcal P$ is not nilpotent as a subalgebra. Consequently, a sum of hypo-nilpotent ideals need not be hypo-nilpotent.
\end{example}

Whenever a finite-dimensional Poisson $n$-Lie algebra contains a hypo-nilpotent ideal, finite dimensionality ensures that this ideal is contained in a maximal one. The following result relates such maximal ideals to the nilradical.

\begin{pro}\label{prop5.12}
Let $\mathcal P$ be a finite-dimensional solvable Poisson $n$-Lie algebra over an algebraically closed field of characteristic zero. Then $\operatorname{Nil}(\mathcal P)$ is contained in every maximal hypo-nilpotent ideal of $\mathcal P$.
\end{pro}

\begin{proof}
Let $\mathcal H$ be a maximal hypo-nilpotent ideal of $\mathcal P$. Since $\mathcal H$ is nilpotent as a Poisson $n$-Lie subalgebra, both $\mathcal H_A$ and $\mathcal H_L$ are nilpotent as subalgebras. If $\mathcal H$ were nilpotent as an ideal of $\mathcal P_L$, Lemma~\ref{lem5.2} would imply that it is nilpotent as an ideal of $\mathcal P$, a contradiction. Thus $\mathcal H$ is a hypo-nilpotent ideal of the $n$-Lie algebra $\mathcal P_L$.

Choose a maximal hypo-nilpotent ideal $\mathcal K$ of $\mathcal P_L$ containing $\mathcal H$. The corresponding result for $n$-Lie algebras \cite{BaiR} gives
$$\mathcal N_L\subseteq\mathcal K,$$
where $\mathcal N_L$ denotes the nilradical of $\mathcal P_L$. By Corollary~\ref{cor5.11}, $\mathcal N_L=\operatorname{Nil}(\mathcal P)$. Moreover,
$$\mathcal P\cdot\mathcal K\subseteq\mathcal P\cdot\mathcal P\subseteq\mathcal P^2\subseteq\operatorname{Nil}(\mathcal P)\subseteq\mathcal K.$$
Hence $\mathcal K$ is an ideal of the Poisson $n$-Lie algebra $\mathcal P$. Its associative part is nilpotent because $\mathcal P_A$ is nilpotent, and its $n$-Lie part is nilpotent as a subalgebra. Proposition~\ref{prop5.3} shows that $\mathcal K$ is nilpotent as a Poisson $n$-Lie subalgebra. Since it is not nilpotent as an $n$-Lie ideal, it is a hypo-nilpotent ideal of $\mathcal P$. The maximality of $\mathcal H$ now gives $\mathcal K=\mathcal H$, and therefore $\operatorname{Nil}(\mathcal P)\subseteq\mathcal H$.
\end{proof}

Let $\mathcal P$ be non-abelian and satisfy the hypotheses of Proposition~\ref{prop5.12}, and let $\mathcal H$ be a maximal hypo-nilpotent ideal. Propositions~\ref{prop5.8} and \ref{prop5.12} yield
\begin{equation}\label{eq5.10}
0\neq\mathcal P^2\subseteq\operatorname{Nil}(\mathcal P)\subsetneq\mathcal H\subsetneq\mathcal P.
\end{equation}
Choose a vector-space complement $\mathcal Q$ of $\mathcal H$ in $\mathcal P$. We call $\mathcal P=\mathcal H\oplus\mathcal Q$ a solvable extension of $\mathcal H$. 

We next examine how the maximality of $\mathcal H$ constrains the inner derivations involving elements of the complement $\mathcal Q$.

\begin{thm}\label{thm5.13}
Let $\mathcal P=\mathcal H\oplus\mathcal Q$ be as above. For every nonzero $x\in\mathcal Q$, there exist $m_1,\ldots,m_{n-2}\in\mathcal H$ such that $Q_{(x,m_1,\ldots,m_{n-2})|\mathcal H}$ is not nilpotent.
\end{thm}

\begin{proof}
Suppose that some nonzero $x\in\mathcal Q$ satisfies
$$Q_{(x,y_1,\ldots,y_{n-2})|\mathcal H}\quad\text{nilpotent for all }y_1,\ldots,y_{n-2}\in\mathcal H.$$
Set $\mathcal V=\mathcal H\oplus\mathbb F x$. By \eqref{eq5.10}, $\mathcal P^2\subseteq\mathcal H$, and therefore $\mathcal V$ is an ideal of $\mathcal P$.

Since $\mathcal H$ is nilpotent as an $n$-Lie subalgebra, $Q_{Y|\mathcal H}$ is nilpotent for every $Y\in\bigwedge^{n-1}\mathcal H$. In addition, every inner derivation of $\mathcal V$ maps $\mathcal V$ into $\mathcal H$. It follows that the operators
$$Q_{Y|\mathcal V},\qquad Y\in\bigwedge^{n-1}\mathcal H, \qquad Q_{(x,y_1,\ldots,y_{n-2})|\mathcal V},\qquad y_i\in\mathcal H,$$
are nilpotent.

The Lie algebra of inner derivations of the solvable $n$-Lie algebra $\mathcal V$ is solvable; see \cite[Proposition~3]{Kasymov}. Lie's theorem therefore triangularizes all these operators simultaneously. Since the displayed generators are nilpotent, they are simultaneously strictly upper triangular. Consequently, every inner derivation of $\mathcal V$ is nilpotent. The Engel theorem for $n$-Lie algebras \cite[Theorem~3]{Kasymov} implies that $\mathcal V_L$ is nilpotent.

The associative algebra $\mathcal V_A$ is nilpotent because it is a subalgebra of the nilpotent algebra $\mathcal P_A$. Proposition~\ref{prop5.3} now shows that $\mathcal V$ is nilpotent as a Poisson $n$-Lie subalgebra. It cannot be nilpotent as an ideal of $\mathcal P$, since otherwise its subideal $\mathcal H$ would also be nilpotent as an ideal. Thus $\mathcal V$ is a hypo-nilpotent ideal properly containing $\mathcal H$, contradicting the maximality of $\mathcal H$.
\end{proof}

We next consider the stronger situation in which one of these inner derivations is invertible on $\mathcal P^2$.

\begin{pro}\label{prop5.14}
Let $\mathcal P$ be a finite-dimensional complex solvable Poisson $n$-Lie algebra. Suppose that, for some $x,m_1,\ldots,m_{n-2}\in\mathcal P$, the restriction 
$$Q_{(x,m_1,\ldots,m_{n-2})|\mathcal P^2}$$
is invertible. Then $\mathcal P\cdot\mathcal P=0$.
\end{pro}

\begin{proof}
Write $m=(m_1,\ldots,m_{n-2})$ and $D=Q_{(x,m)}$. Since $D(x)=0$, the Leibniz identity gives
$$D(x^2)=2xD(x)=0.$$
As $x^2\in\mathcal P^2$ and $D|_{\mathcal P^2}$ is injective, we obtain $x^2=0$.

For every $v\in\mathcal P^2$, the Leibniz identity in the first entry of the bracket gives
$$2xD(v)=Q_{(x^2,m)}(v)=0.$$
The surjectivity of $D|_{\mathcal P^2}$ therefore implies $x\mathcal P^2=0$. If $a\in\mathcal P$, then $xa\in\mathcal P^2$ and 
$$D(xa)=D(x)a+xD(a)=0,$$
because $D(a)\in[\mathcal P,\ldots,\mathcal P]\subseteq\mathcal P^2$. Injectivity on $\mathcal P^2$ yields $xa=0$, so $x\mathcal P=0$.

Finally, let $a,b\in\mathcal P$. Since $xb=0$, another application of the Leibniz identity gives
$$0=Q_{(xb,m)}(a)=xQ_{(b,m)}(a)+bD(a)=bD(a),$$
where the first term vanishes because $Q_{(b,m)}(a)\in\mathcal P^2$. Similarly, $aD(b)=0$. Hence
$$D(ab)=D(a)b+aD(b)=0.$$
Since $ab\in\mathcal P^2$ and $D|_{\mathcal P^2}$ is injective, it follows that $ab=0$. Thus $\mathcal P\cdot\mathcal P=0$.
\end{proof}

\begin{rmk}
For $n=2$, an analogous statement was proved in \cite{Amir2} under the stronger assumption that $Q_x$ be invertible on $\operatorname{Nil}(\mathcal P)$. Proposition~\ref{prop5.14} requires invertibility only on the generally smaller ideal $\mathcal P^2$.
\end{rmk}

To establish the existence of an example satisfying the conditions of Proposition~\ref{prop5.14}, we first introduce the notion of a torus in an $n$-Lie algebra.

\begin{defn} A {\it torus} on an $n$-Lie algebra $\mathcal{L}$ is a commutative subalgebra of the Lie algebra $\operatorname{Der}(\mathcal {L})$ which consists of semisimple endomorphisms. A torus is said to be {\it maximal} if it is not contained strictly in any other torus. 
\end{defn}

Assume that $\mathcal L$ is finite-dimensional over an
algebraically closed field of characteristic zero.
The elements of a torus can be simultaneously diagonalized. By Mostow's conjugacy theorem for maximal completely reducible
subgroups \cite{Mostow}, together with the standard conjugacy
theorem for maximal tori in linear algebraic groups, any two
maximal tori $\mathcal T_1$ and $\mathcal T_2$ of $\mathcal L$
are conjugate under $\operatorname{Aut}(\mathcal L)$. Thus all maximal tori have the same dimension, called the
\emph{rank} of $\mathcal L$ and denoted by
$\operatorname{rank}(\mathcal L)$.
The following example uses the diagonal torus of an abelian
$n$-Lie algebra.

\begin{example}
Let $n\geq3$ and $m\geq n-1$, and let $\mathcal N$ be an
$m$-dimensional abelian complex $n$-Lie algebra with basis
$\{e_1,\ldots,e_m\}$. The diagonal derivations
$$
t_i(e_j)=\delta_{ij}e_j,\qquad 1\leq i,j\leq m,
$$
span a maximal torus $\mathcal T$ of $\mathcal N$.

Consider the vector space
$
\mathcal L=\mathcal N\oplus\mathcal Q,
\quad
\mathcal Q=\operatorname{span}\{t_i\mid n-1\leq i\leq m\},
$
with the nonzero brackets
$$
[t_i,e_1,\ldots,e_{n-2},e_i]=e_i,
\qquad n-1\leq i\leq m.
$$
All other brackets of basis elements are zero, apart from
those determined by skew-symmetry.
A direct check of the Filippov identity shows that these
brackets define an $n$-Lie algebra. Moreover,
$$
\mathcal L^2=[\mathcal L,\ldots,\mathcal L]
=\operatorname{span}\{e_{n-1},\ldots,e_m\},
\qquad
[\mathcal L^2,\mathcal L^2,\mathcal L,\ldots,\mathcal L]=0.
$$
Thus $\mathcal L$ is solvable. The defining brackets also
show that $\mathcal N$ is a maximal hypo-nilpotent ideal
of $\mathcal L$ and
$
\operatorname{Nil}(\mathcal L)=\mathcal L^2.
$
For $y=\sum_{i=n-1}^{m}t_i$, the operator
$
\left.Q_{(y,e_1,\ldots,e_{n-2})}\right|_{\mathcal L^2}
$
is invertible.

Now let $\mathcal P$ be a solvable Poisson $n$-Lie algebra
such that $\mathcal P_L=\mathcal L$ and
$\mathcal P\cdot\mathcal P
\subseteq[\mathcal P,\ldots,\mathcal P].
$
Then $\mathcal P^2=\mathcal L^2$, so
$
\left.Q_{(y,e_1,\ldots,e_{n-2})}\right|_{\mathcal P^2}
$
is invertible. Proposition~\ref{prop5.14} therefore gives
$\mathcal P\cdot\mathcal P=0$.
\end{example}

We now turn to ideals arising from the associative product. The following result applies to arbitrary Poisson $n$-Lie algebras and adapts an argument from~\cite{Amir3}.

\begin{pro}
Let $\mathcal P$ be a Poisson $n$-Lie algebra over a field $\mathbb F$. For $a\in\mathcal P$ and $\lambda\in\mathbb F$, the generalized eigenspace
$$\mathcal P_{a,\lambda}=\{z\in\mathcal P\mid  (P_a-\lambda\operatorname{id}_{\mathcal P})^k(z)=0  \text{ for some }k\geq1\}$$
is an ideal of $\mathcal P$.
\end{pro}

\begin{proof}
Set $T=P_a-\lambda\operatorname{id}_{\mathcal P}$. Since all multiplication operators commute,
$$T^kP_y=P_yT^k$$
for every $y\in\mathcal P$ and $k\geq1$. Hence $\mathcal P\cdot\mathcal P_{a,\lambda}\subseteq\mathcal P_{a,\lambda}$.

Let $Y\in\bigwedge^{n-1}\mathcal P$. The Leibniz identity gives the operator relation
$$TQ_Y=Q_YT-P_{Q_Y(a)}.$$
Since $P_{Q_Y(a)}$ commutes with $T$, induction on $k$ yields
\begin{equation}\label{eq5.11}
T^kQ_Y=Q_YT^k-kP_{Q_Y(a)}T^{k-1},\qquad k\geq1.
\end{equation}
If $T^k(z)=0$, then \eqref{eq5.11}, with $k+1$ in place of $k$, gives $T^{k+1}Q_Y(z)=0$. Therefore
$$[\mathcal P,\ldots,\mathcal P,  \mathcal P_{a,\lambda}]\subseteq\mathcal P_{a,\lambda},$$
and $\mathcal P_{a,\lambda}$ is an ideal.
\end{proof}

Finally, the usual Peirce argument for associative algebras extends to the present setting. Suppose that $\mathcal P$ is finite-dimensional over an algebraically closed field of characteristic zero and that $\mathcal P_A$ contains a nonzero idempotent $e$. The following consequences are proved exactly as in the binary Poisson case \cite{Amir2}:
\begin{itemize}
\item $e$ belongs to the center of $\mathcal P_L$;
\item if $\mathcal P_L$ is solvable and has trivial center, then $\mathcal P_A$ is nilpotent;
\item if $\mathcal P_L$ is solvable, then the Peirce decomposition associated with $e$ yields a split Poisson $n$-Lie algebra. In particular, if $\mathcal P_L$ is non-split, then $\mathcal P_A$ is nilpotent.
\end{itemize}

\

{\bf Funding} {{\small{\ The third author was supported by Basic Research Program of Jiangsu (BK20251784).}}}
 \\[3mm]
{\bf Data availability} {{\small{\ Data sharing is not applicable to this article as no datasets were generated or analysed during
the current study.}}}
\\[3mm]
{\bf Conflict of interest} {{\small{\ The authors have no competing interests to declare that are relevant to the content of this
article.}}}

\end{document}